\documentclass[letterpaper,10pt]{amsart}
\usepackage{indentfirst} % Indent first paragraph after section header
\usepackage{amssymb}
\usepackage{amsmath}
\usepackage{mathabx} % makes many symbols (as $\leq$) more beautiful
\usepackage{amsthm}
\usepackage{thmtools}
\usepackage{bbm}             %to use $\mathbbm{1}$
\usepackage[colorlinks=true]{hyperref}  % Links in the pdf
\usepackage[usenames,dvipsnames]{xcolor} % With XCOLOR, printed quality is good, (but with COLOR it's bad)
\usepackage{tikz}

\declaretheorem{theorem}

\declaretheorem{lemma}
\declaretheorem{corollary}
\declaretheorem{proposition}

\declaretheorem{definition}

\declaretheorem{remark}
\declaretheorem[name=Acknowledgements,numbered=no]{ack}

\newcommand{\M}{\mathcal{M}}
\newcommand{\R}{\mathbb{R}}
\newcommand{\Z}{\mathbb{Z}}
\newcommand{\N}{\mathbb{N}}
\newcommand{\E}{\mathbb{E}}
\renewcommand{\P}{\mathbb{P}}

\def\c{{\bf C}}
\def\e{\epsilon}
\def\phi{\varphi}
\def\R{{\mathbb R}}

\def\N{{\mathbb N}}
\def\Z{{\mathbb Z}}

\def\E{{\mathbb E}}

\def\c{{\mathcal C}}
\def\P{{\mathbb P}}
\def\p{{\mathcal P}}
\def\Q{{\mathcal Q}}

\def\M{{\mathcal M}}

\def\X{{\mathcal X}}

\def\diam{\mbox{\rm diam} }

\def\le{\leqslant}
\def\ge{\geqslant}

\def\M{\mathcal{M}}

\begin{document}

\title{Rate distortion theory, metric mean dimension and measure theoretic entropy}
\date{\today}

\author[A. Velozo]{Anibal Velozo}  \address{Princeton University, Princeton NJ 08544-1000, USA.}
\email{avelozo@math.princeton.edu}

\author[R. Velozo]{Renato Velozo} \address{Facultad de Matem\'aticas,
Pontificia Universidad Cat\'olica de Chile (PUC), Avenida Vicu\~na Mackenna 4860, Santiago, Chile}
\email{ravelozo@mat.uc.cl}

\begin{abstract} We prove a variational principle for the metric mean dimension analog to the one in \cite{LT}. Instead of using the rate distortion function we use the function $h_\mu(\e,T,\delta)$ that is closely related to the entropy $h_\mu(T)$ of $\mu$. Our formulation has the advantage of being, in the authors opinion, more natural  when doing computations.  As a corollary we obtain a proof of the standard variational principle. We also obtain some relations between the rate distortion function with our function $\widetilde{h}_\mu(\e,T,\delta)$, a modification of $h_\mu(\e,T,\delta)$ when replacing the dynamical metrics with the average dynamical metrics. Using our methods we also reprove the main result in \cite{LT}.  We will explain how to construct homeomorphisms on closed manifolds with maximal metric mean dimension. We end this paper with some questions that naturally arise  from this work. 
\end{abstract}

\maketitle

\section{Introduction}
The \emph{topological entropy} is a fundamental quantity that allows us to quantify the chaoticity of a dynamical system. If the ambient space is compact and the dynamics is Lipschitz, then the topological entropy is finite. On the other hand, if the dynamics is just continuous, the topological entropy might be infinite. In fact, K. Yano proved in \cite{yan} that on a closed manifold of dimension at least two the topological entropy is infinite for generic homeomorphisms. It is then natural to consider a dynamical quantity that distinguishes systems with infinite topological entropy. The \emph{mean dimension} is a meaningful quantity when the topological entropy is infinite. This invariant was introduced by Gromov in \cite{gro}, and further studied by E. Lindenstrauss and B. Weiss in \cite{LW}. This invariant has found many applications to embedding problems, in other words to the problem of when a dynamical system can be embedded into another or not, see for instance \cite{l},\cite{glt} and references therein.  In this paper we will mainly focus on the \emph{metric mean dimension}. This is an invariant of the dynamical system defined in \cite{LW}, which in contrast with the topological entropy  might depend on the metric on the ambient space. For completeness we proceed to define the relevant quantities. Let $(\X,d)$ be a compact metric space and $T:\X\to \X$ a continuous map. Define $N_d(n,\e)$ as the maximal cardinality of a $(n,\e)$-separated subset of $\X$ and $$S(\X,d,\e)=\limsup_{n\to \infty}\dfrac{1}{n}\log N_d(n,\e).$$ For precise definitions see Section \ref{prelim}. The topological entropy is defined as $$h_{top}(\X,T)=\lim_{\e\to 0} S(\mathcal{X},d,\e).$$
The topological entropy is known to be independient of the compatible metric $d$ since $\X$ is compact. The upper metric mean dimension is defined as
 $$\overline{mdim}(\X,d,T)=\limsup_{\e\to 0}\dfrac{S(\mathcal{X},d,\e)}{|\log\e|}.$$

Recently E. Lindenstrauss and M. Tsukamoto in \cite{LT} established a variational principle for the metric mean dimension. In this formula the role of the measure theoretic entropy is replaced by the rate distortion function $R_\mu(\e)$ (for its definition see Section \ref{prelim}). Let $\M_T(\X)$ be the space of $T$-invariant probability measures on $\X$. They proved the following result.

\begin{theorem} \label{teolin}Suppose $(\X,d)$ satisfy Condition 1.2 in \cite{LT}. Then 
 $$\overline{mdim}(\X,d,T)=\limsup_{\e\to 0}\dfrac{\sup_{\mu\in\M_T(\X)} R_\mu(\e)}{|\log\e|}.$$
\end{theorem}
For the meaning of Condition 1.2 see Definition \ref{cond}. In this paper we will prove an analog formula where instead of the rate distortion function we use some terms that appear in Katok's entropy formula. This has the advantage to simplify some computations and to make them, in the authors opinion, more intuitive. For $\mu\in \M_T(\X)$ define $N_\mu(n,\e,\delta)$ as the minimum number of $(n,\e)$-dynamical balls needed to cover a set of measure strictly bigger than $1-\delta$. Then define $h_\mu(\e,T,\delta)=\limsup_{n\to\infty} \frac{1}{n}\log N_\mu(n,\e,\delta)$. It was proven by A. Katok in \cite{ka} that $h_\mu(T)=\lim_{\e\to 0} h_\mu(\e,T,\delta)$ for every ergodic measure $\mu$ and any $\delta\in (0,1)$, where $h_\mu(T)$ is the measure theoretic entropy of $\mu$. One of the main results of this paper is the following theorem. 

\begin{theorem}\label{theo1} Let $(\X,d)$ be a compact metric space and $T:\X\to\X$ continuous. Then 
 $$\overline{mdim}(\X,d,T)=\lim_{\delta\to 0}\limsup_{\e\to 0}\dfrac{\sup_{\mu\in\M_T(\X)}h_\mu(\e,T,\delta)}{|\log\e|}.$$
\end{theorem}

From the proof of Theorem \ref{theo1} we also recover a proof of the standard variational principle. An easy application of Theorem \ref{theo1} is the computation $$mdim(([0,1]^n)^\Z,d_T,T)=n,$$ where the transformation $T$ is the shift map and $d_T(x,y)=\sum_{k\in\Z} \frac{1}{2^{|k|}}d(x_k,y_k)$, where $x=(...,x_{-1},x_0,x_1...)$, $y=(...,y_{-1},y_0,y_1...)$ and $d$ is the standard metric on $[0,1]^n$. We also investigate the connection between the rate distortion function $R_\mu(\e)$ and our replacement $h_\mu(\e,T,\delta)$. For reasons that will be clear to the reader the function $R_\mu(\e)$ is closely related to $\widetilde{h}_\mu(\e,T,\delta)$, where instead of using $(n,\e)$-dynamical balls we use $(n,\e)$-average dynamical balls.  We are in particular able to reprove Theorem \ref{teolin}. We also prove that the rate distortion function recovers the measure theoretic entropy in the ergodic case. 

\begin{theorem} \label{theo2} Let $(\X,d)$ be a compact metric space and $\mu\in \M_T(\X)$ an ergodic measure. Then
$$h_\mu(T)=\lim_{\e\to 0} R_\mu(\e)=\widetilde{h}_\mu(T,\delta).$$
\end{theorem}

We then obtain a result analog to the one in Theorem \ref{theo1} when we replace the $(n,\e)$-dynamical balls with the $(n,\e)$-average dynamical balls, i.e. when replacing $h_\mu(\e,T,\delta)$ by $\widetilde{h}_\mu(\e,T,\delta)$. We use this formula to reprove Theorem \ref{teolin}.

The paper is organized as follows. In Section \ref{prelim} we recall some basic definitions from ergodic and information theory. In Section \ref{sec3} we prove Theorem \ref{theo1} and we compute the metric mean dimension of the shift over $([0,1]^n)^\Z$. In Section \ref{sec3} we also obtain a proof of the standard variational principle. In Section \ref{sec4} we discuss the connections between the rate distortion function $R_\mu(\e)$ and the function $\widetilde{h}_\mu(\e,T,\delta)$. We use the results in Section \ref{sec4} to reprove Theorem \ref{teolin}. In Section \ref{sec5} we discuss how generic is for a continuous map or homeomorphism on a manifold to have positive or maximal metric mean dimension. In Section \ref{sec6} we make some final remarks and state some natural questions, we also suggest the definition of what should be the metric mean dimension of a measure.

\begin{ack}
The authors would like to thank G. Iommi and M. Tsukamoto for their interest in this work. The first author would like to thanks to his advisor G. Tian for his constant support and encouragements. The last part of this paper was written when the first author was visiting Pontificia Universidad Cat\'olica de Chile, he is very grateful to G. Iommi for the invitation. Finally, the second author would like to thanks to his advisor J. Bochi  for his continued guidance and encouragements.
\end{ack}

\section{Preliminaries}\label{prelim}
Let  $(\X,d)$ be a metric space and  $T:\X\to\X$ a continuous map. We refer to the triple $(\X,d,T)$ a \emph{dynamical system}. We will not always assume that $\X$ is compact, we will specify when that assumption is required. The metric $d$ induces a topology on $\X$ and this topology endows $\X$ with the borelian $\sigma$-algebra. Any measure on $\X$ is assumed to be defined on the borelian $\sigma$-algebra of $(\X,d)$. We will use standard concepts in ergodic theory, for completeness we briefly define the notions more relevants to this work.  A probability measure $\mu$ on $\X$ is said to be invariant under $T$ or \emph{$T$-invariant} if for any measurable set $A$ we have $\mu(A)=\mu(T^{-1}A)$. A probability $T$-invariant measure $\mu$ is said to be \emph{ergodic} if any measurable set $A$ satisfying $T^{-1}A=A$ has measure zero or one. To a partition $\p=\{\p_1,...,\p_m\}$ of $\X$ we can assign the value $$H_\mu(\p)=-\sum_{i=1}^m \mu(\p_i)\log\mu(\p_i).$$
Observe that everytime we have a measurable map $Z:\X\to W$ with finite image we can associate a partition on $\X$, the \emph{preimage partition} of $Z$ (which is a finite partition since the map has finite image). In this case we denote by $H_\mu(Z)$ to the entropy of the preimage partition of $Z$. Let $\p\vee \mathcal{Q}$ be the refinement of the partitions $\p$ and $\mathcal{Q}$. Given two random variables with finite image $Z_1:\X\to W_1$ and $Z_2:\X\to W_2$ we can define the \emph{mutual information of $Z_1$ and $Z_2$} as $$I_\mu(Z_1,Z_2)=H_\mu(Z_1)+H_\mu(Z_2)-H_\mu(Z_1,Z_2),$$ where $H_\mu(Z_1,Z_2)$ is the entropy of the refinement of the preimage partitions of $Z_1$ and $Z_2$. We define the partitions $\p^n:=\p\vee T^{-1}\p\vee ...\vee T^{-(n-1)}\p$ and the quantity 
 $$h_\mu(\p,T)=\inf_{n}\dfrac{1}{n}H_\mu(\p^n).$$
We define the \emph{entropy of the measure} $\mu$ as $$h_\mu(T)=\sup_\p h_\mu(\p,T),$$
where the supremum runs over finite partitions of $\X$. For more detailed information about the entropy of a measure and its relevance in ergodic theory we refer the reader to \cite{W}. Using the map $T$ and the metric $d$ we can define two new families of metrics on $\X$ by the following expressions  
$$d_n(x,y)=\max_{k\in \{0,...,n-1\}}d(T^kx,T^ky),$$
$$\widetilde{d}_n(x,y)=\dfrac{1}{n}\sum_{k=0}^{n-1} d(T^kx,T^ky).$$
We denote a ball of radius $\e$ in the metric $d_n$  as a \emph{$(n,\e)$-dynamical ball} and a ball of radius $\e$ in the metric $\widetilde{d}_n$  as a \emph{$(n,\e)$-average dynamical ball}.  The following definition allows us to state Theorem \ref{kat} in a cleaner way and also motivates Definition \ref{ent2}.

\begin{definition} \label{ent} Let $\mu$ be a $T$-invariant probability measure. For $\delta\in (0,1)$, $n\in \N$ and $\epsilon>0$ we define $N_\mu(n,\epsilon,\delta)$ to be the minimum number of $(n,\epsilon)$-dynamical balls needed to cover a set of $\mu$-measure strictly bigger than $1-\delta$. Set  $$h_\mu(\epsilon,T,\delta)=\limsup_{n\to \infty}\dfrac{1}{n}\log N_\mu(n,\epsilon,\delta).$$ 
\end{definition}

%\begin{remark} The $\limsup$ above is indeed independient of $\delta$. This is an easy consequence of the ergodicity of $\mu$, i.e. from the fact that for any measurable set $A$ with $\mu(A)>0$, we have  $\lim_{n\to\infty}\mu(A\cup T^{-1}A\cup ...\cup T^{-n}A)=1$. This remark works for Definition \ref{ent2}  as well. 
%\end{remark}

%\begin{lemma} Let $\delta,\delta'\in (0,1)$. Then $$\limsup_{n\to\infty}\dfrac{1}{n}\log N_\mu(n,r,\delta)=\limsup_{n\to\infty}\dfrac{1}{n}\log N_\mu(n,r,\delta').$$
%\end{lemma}
%\begin{proof} We can assume $\delta<\delta'$ or equivalently $1-\delta'<1-\delta$. This immediately implies $$N_\mu(n,r,\delta')\le N_\mu(n,r,\delta).$$
%Now start with a set $A$ with $\mu(A)\ge 1-\delta'$ such that $N(A,n,\e)=N_\mu(n,\e,\delta')$.  By ergodicity of $\mu$ we can find $N$ such that $\mu(A\cup T^{-1}A\cup ... \cup T^{-N}A)>1-\delta$. Pick one of the $(n,\e)$-dynamical balls that cover $A$ optimally, denote it by $B_n(x,\e)$. Denote by $M$ the minimum number of balls of radius $\e$ needed to cover $\X$. Therefore $T^{-k}B_n(x,\e)$ can be covered by  $M^k$ $(k,\e)$-dynamical balls, which implies that  it can also be covered by $M^k$  $(n,\e)$-dynamical balls. We conclude that $A\cup T^{-1}A\cup ... \cup T^{-N}A$ can be covered by $N(A,n,\e)(1+M+...+M^N)$ $(n,\e)$-dynamical balls. Therefore $$N_\mu(n,r,\delta)\le N_\mu(n,\e,\delta')(M+1)^N,$$
%which implies the desired conclusion.
%\end{proof}

The following theorem was proven in \cite{ka}. 

\begin{theorem}\label{kat} Let $(\X,d)$  be a compact metric space, $T:\X\to \X$ a continuous transformation and $\mu$  an ergodic $T$-invariant probability measure. Then 
$$h_\mu(T)=\lim_{\epsilon\to 0} h_\mu(\epsilon,T,\delta),$$
where $h_\mu(T)$ is the measure theoretic entropy of $\mu$. In particular the limit above does not depend on $\delta\in (0,1)$. 
\end{theorem}

\begin{remark} \label{rem2} Theorem \ref{kat} is proven by establishing inequalities between LHS and RHS. The inequality LHS $\ge$ RHS can be stated more precisely as 
$$h_\mu(\mathcal{P},T)\ge \limsup_{n\to\infty}\dfrac{1}{n}\log N_\mu(n, \epsilon,\delta),$$
where $\p$ is any partition of diameter less than $\epsilon$. This fact will be used in the proof of Proposition \ref{prop1}. We emphasize this inequality is independient of $\delta$. 

\end{remark}

The following definition mimics Definition \ref{ent} when we use $\widetilde{d}_n$ instead of $d_n$. We use Katok's formula to motivate the definition of the \emph{average measure theoretic entropy} of an ergodic measure $\mu$.

\begin{definition} \label{ent2} Let $\mu$ be a $T$-invariant probability measure  on $\X$.  We define $$\widetilde{h}_\mu(\epsilon, T,\delta)=\limsup_{n\to \infty}\dfrac{1}{n} \log \widetilde{N}_\mu(n,\epsilon,\delta),$$
where  $\delta\in (0,1)$ and $\widetilde{N}_\mu(n,\epsilon,\delta)$ is the minimum number of $(n,\epsilon)$-average dynamical balls needed to cover a set of $\mu$-measure  strictly bigger than $1-\delta$. We define the \emph{average measure theoretic entropy} of $\mu$ as the limit $$\widetilde{h}_\mu(T,\delta)=\lim_{\epsilon\to 0} \widetilde{h}_\mu(\epsilon, T,\delta).$$

\end{definition}
One of the goals of this paper is to establish relations between the \emph{rate distortion function} $R_\mu(\e)$ recently used by E. Lindenstrauss and M. Tsukamoto \cite{LT} in the study of the metric mean dimension, and the quantities defined above.  We start by recalling the definition of $R_\mu(\e)$.
\begin{definition} \label{cond} We say that the pair $(X,Y)$ satisfies condition $(*)_{n,\e}$ if the following properties are satisfied.
\begin{enumerate}
\item There exists a probability space $(\Omega,\P)$ such that $X:\Omega\to \X$ and $Y=(Y_0,...,Y_{n-1}):\Omega\to \X^n$ are measurable functions. 
\item The measure $X_*\P$ is a $T$-invariant probability measure on $\X$.
\item $\E(\frac{1}{n}\sum_{k=0}^{n-1}d(T^kX,Y_k))\le \epsilon$.

\end{enumerate}
We say that the pair $(X,Y)$ satisfies condition $(*)_{n,\e,\mu}$ if moreover $X_*\P=\mu$. 

\end{definition}

\begin{definition} Given a $T$-invariant probability measure $\mu$ on $\X$, define the \emph{rate distortion function} as $$R_\mu(\e)=\inf \dfrac{1}{n}I(X,Y),$$ 
where the infimum runs over pairs $(X,Y)$ satisfying condition $(*)_{n,\e,\mu}$.
\end{definition}
In the definition above $I(X,Y)$ is the mutual information of the random variables $X$ and $Y$. If $X$ and $Y$ have finite image this was defined in the second paragraph of the introduction. For the general case (assuming $\X$ is a standard probability space) we refer the reader to Section 5.5 in \cite{G}.

\section{The variational principle}  \label{sec3}
In this section we prove Theorem \ref{theo1} and the standard variational principle. Define $N_d(n,\e)$ as the maximal cardinality of a $(n,\e)$-separated set in $(\X,d)$ and $$S(\X,d,\e)=\limsup_{n\to \infty}\dfrac{1}{n}\log N_d(n,\e).$$ In case the dynamical system has been specified, we will frequently use the simplified notations $S(\e)=S(\X,d,\e)$. We start with the following lemma which is important for our results. 

\begin{lemma} \label{1111} Assume $(\X,d)$ is compact and $T:\X\to \X$ a continuous map. Given $\e>0$, there exists a $T$-invariant measure $\mu_\e$ such that $$h_{\mu_\e}(\e,T,\delta)\ge  S(\X,d,2\e)-3\delta S(\X,d,\e).$$
\end{lemma}

\begin{proof} Let  $E_n=\{x_1,...,x_{N_d(n,\e)}\}$ be a maximal collection of $(n,\e)$-separated points in $\X$.  Define $$\sigma_n=\dfrac{1}{|E_n|}\sum_{x\in E_n} \delta_x,$$
where $\delta_x$ is the probability measure supported at $x$.  Then define   
$$\overline{\sigma}_n=\dfrac{1}{n}\sum_{k=0}^{n-1} T^k_*\sigma_n.$$
Consider a subsequence $\{n_k\}_{k\in \N}$ such that $$S(\X,d,\e)=\lim_{k\to \infty}\dfrac{1}{n_k}\log N_d(n_k,\e).$$ By standard arguments we can find a subsequence of $\{n_k\}$, which we still denote as $\{n_k\}$, such that $\{\overline{\sigma}_{n_k}\}_{k\in\N}$ converges to a $T$-invariant probability measure $\mu$. We can arrange the sequence such that $\lim_{k\to \infty} \frac{n_k}{k}=\infty$. Let $K$ be a subset of $\X$ with $\mu(K)> 1-\delta$, and $N(K,n,\e)=N_\mu(n,\e,\delta)$, where $N(K,n,\e)$ is defined to be the minimum number of $(n,\e)$-dynamical balls needed to cover $K$. We can assume that $K$ is open in $\X$. There exists $k_0$ such that for every $k\ge k_0$ we have $\overline{\sigma}_{n_k}(K)> 1-\delta$. Let $$L_n=\{(i,j)\in\N^2: 0\le i\le n-1, 1\le j\le N_d(n,\e)\}.$$ We assign to each point in $L_n$ either a $0$ or $1$ in the following way. If $T^ix_j\in K$ then assign $1$ to the point $(i,j)$,  assign $0$ to $(i,j)$ otherwise. By definition of the measure $\overline{\sigma}_{n_k}$ we know that the number of ones in $L_{n_k}$ is bigger or equal than $n_k N_d(n_k,\e)(1-\delta)$. For $s>1$ we define $L_{n_k}(s)$ as the set of points in $L_{n_k}$ with first coordinate in the interval $[[n_k/k],n_k-[n_k/s]-1]$. The number of ones in $L_{n_k}(s)$ is at least $n_k N_d(n_k,\e)(1-\delta-\frac{1}{n_k}[\frac{n_k}{s}]-\frac{1}{n_k}[\frac{n_k}{k}])\ge n_k N_d(n_k,\e)(1-\delta-\frac{1}{s}-\frac{1}{k})$. From now on we assume $s>\frac{1}{1-2\delta-\frac{1}{k}}$, in particular $1-\delta-\frac{1}{s}-\frac{1}{k}>\delta$. We will moreover assume $s<\frac{1}{1-3\delta}$, this can be done if we assume $k$ is sufficiently large so that $\delta>1/k$.    We conclude that the number of ones in $L_{n_k}(s)$ is at least  $n_k N_d(n_k,\e)\delta$. Since $L_{n_k}$ has $(n_k-[n_k/s]-[n_k/k])$ columns, there exists an index $m_k$ such that the $m_k$th column has at least $n_k N_d(n_k,\e)\delta/(n_k-[n_k/s]-[n_k/k])$ ones and $[n_k/k]\le m_k< n_k-[n_k/s]$. Observe that $\X$ can be covered by at most $N_d(m_k,\e/2)$ $(m_k,\e/2)$-dynamical balls, in particular with $N_d(m_k,\e/2)$ subsets of $d_{m_k}$-diameter smaller than $\e$. Also observe that if $i\ne j$ and $d_{m_k}(x_i,x_j)\le \e$, then $d_{n_k-m_k}(T^{m_k}x_i,T^{m_k} x_j)>\e$, because $d_{n_k}(x_i,x_j)>\e$ by the definition of $E_{n_k}$. We can conclude that there exists an subset $I\subset \{1,....,N_d(n_k,\e)\}$ such that for $i\in I$ we have $T^{m_k}x_i\in K$, and $|I|\ge n_k N_d(n_k,\e)\delta/(n_k-[n_k/s]-[n_k/k])>N_d(n_k,\e)\delta$. In other words there exists a subset $A$ of $\{T^{m_k} x_i\}_{i\in I}$ such that  the diameter of $A$ with respect to $d_{m_k}$ is at most $\e$ and $$|A|\ge \dfrac{ N_d(n_k,\e)\delta}{N_d(m_k,\e/2)  }.$$ This implies that if $a,b\in A$ and $a\ne b$, then $d_{n_k}(a,b)\ge d_{n_k-m_k}(a,b)\ge \e$. Then 
$$N_\mu(n_k,\e/2,\delta)=N(K,n_k,\e/2)\ge \dfrac{ N_d(n_k,\e)\delta}{N_d(m_k,\e/2)  }.$$
Therefore
 $$h_\mu(\e/2,T,\delta)=\limsup_{n\to \infty}\dfrac{1}{n}\log N_\mu(n,\e/2,\delta)\ge \limsup_{k\to\infty} \dfrac{1}{n_k}\log N_d(n_k,\e)-\dfrac{m_k}{n_k}\dfrac{1}{m_k}\log N_d(m_k,\e/2).$$
Recall that by construction $m_k<n_k-[n_k/s]$, then $\frac{m_k}{n_k}\le 1-\frac{1}{s}+\frac{1}{n_k}$. Therefore
\begin{align*} h_\mu(\e/2,T,\delta)&\ge \limsup_{k\to\infty} \dfrac{1}{n_k}\log N_d(n_k,\e)-(1-\dfrac{1}{s}+\dfrac{1}{n_k})\dfrac{1}{m_k}\log N_d(m_k,\e/2).\\
&\ge \limsup_{k\to\infty} \dfrac{1}{n_k}\log N_d(n_k,\e)-(3\delta+\dfrac{1}{n_k})\dfrac{1}{m_k}\log N_d(m_k,\e/2).
\end{align*}
 Also observe that the inequality $[n_k/k]\le m_k$ implies that $m_k\to \infty$ as $k\to \infty$. By construction $S(\X,d,\e)=\lim_{k\to \infty}\dfrac{1}{n_k}\log N_d(n_k,\e)$, therefore 
$$h_\mu(\e/2,T,\delta)\ge  S(\e)-3\delta S(\e/2).$$ \end{proof}
In particular, last lemma proves that $$\sup_\mu h_\mu(\e,T,\delta)\ge S(2\e)-3\delta S(\e),$$
where the supremum runs over the set of $T$-invariant probability measures. This observation will be useful in the next corollary. A straightforward application of Lemma \ref{1111} is a proof of the  standard variational principle and the proof of Theorem \ref{theo1}.

\begin{corollary}[Classical Variational principle] Let $(\X,d)$ be a compact metric space and $T:\X\to\X$ a continuous map. Then $$h_{top}(\X,T)=\sup_{\mu\in\M_T(\X)} h_\mu(T).$$
\end{corollary}
\begin{proof} From Proposition \ref{1111} we know $$\sup_\mu h_\mu(\e,T,\delta)\ge S(2\e)-3\delta S(\e).$$
Let $\mu=\int \mu_x d\mu(x)$ be the ergodic decomposition of $\mu$. The proof of the inequality in Remark \ref{rem2} can be adapted to obtain $$h_\mu(\e,T,\delta)\le \sup_{x\in \X} h_{\mu_x}(\Q)\le \sup_{\mu\text{ ergodic}}h_\mu(T),$$
for any partition $\Q$ of diameter smaller than $\e$. Also observe that Theorem \ref{kat} implies 
$$ h_\mu(T)\le h_{top}(\X,T),$$
for every ergodic measure $\mu$. Finally $$S(2\e)-3\delta S(\e)\le \sup_\mu h_\mu(\e,T,\delta)\le \sup_{\mu\text{ ergodic}}h_\mu(T)\le h_{top}(T).$$
Finally taking $\e\to 0$ and $\delta\to 0$ we get that $$\sup_{\mu\text{ ergodic}}h_\mu(T)=h_{top}(T).$$
\end{proof}

\begin{corollary} Let $(\X,d)$ be a compact metric space and $T:\X\to\X$ a continous map. Then 
 $$\overline{mdim}(\X,d,T)=\lim_{\delta\to 0}\limsup_{\e\to 0}\dfrac{\sup_{\mu\in\M_T(\X)}h_\mu(\e,T,\delta)}{|\log\e|}.$$
\end{corollary}

\begin{proof}
The inequality $LHS\ge RHS$ follows directly from Lemma \ref{1111}. The other inequality follows trivially, because by definition $h_\mu(\e,T,\delta)\le S(\epsilon)$.
\end{proof}

As an application of this formula we will compute the metric mean dimension of $([0,1]^\Z,d_T,T)$. It is well known that $\overline{mdim}( [0,1]^{\Z},d_T,T)=1$, see for instance \cite{LW} and \cite{LT}. We will proceed to verify this fact. We recall that the metric $d_T$ is given by $d_T(x,y)=\sum_{k\in\Z} \frac{1}{2^{|k|}}d(x_k,y_k)$, where $x=(...,x_{-1},x_0,x_1...)$, $y=(...,y_{-1},y_0,y_1...)$ and $d$ is the standard metric on $[0,1]$.

\begin{proposition} The metric mean dimension of the shift map on $[0,1]^\Z$ with the metric $d_T$ is given by
$$\overline{mdim}( [0,1]^{\Z},d_T,T)=1.$$
\end{proposition}

\begin{proof} For $k\ge 1$ consider the set $P_k=\{p_1, p_2, \cdots, p_k\}$, where $p_i=\frac{2i-1}{2k}$. Define $\lambda_k$ as the probability measure on $[0,1]$ that equidistribute the points in $P_k$ and let $\mu_k= (\lambda_k)^{\otimes \Z}$, the product measure on $[0,1]^\Z$. Define $$A_{i_0,i_2,\dots, i_{n-1}}=\{x\in [0,1]^{\Z}: x_0=p_{i_0}, \dots , x_{n-1}=p_{i_{n-1}}\}.$$ 
We will need the following fact.  
\begin{lemma}
Let $r<1/2k$ and $q \in [0,1]^{\Z}$. Then there is a unique set $A_{i_0,i_2,\dots, i_{n-1}}$ such that $$supp (\mu_k) \cap B_n(q,r) \subset A_{i_0,i_2,\dots, i_{n-1}}.$$
\end{lemma}

\begin{proof}
Let $x \in supp(\mu_k) \cap B_n(q,r)$. By definition $$d(x_j,q_j)\le d(T^jx,T^jq)\le d_n(x,q)\le r<\frac{1}{2k},\quad \forall j \in\{0,\cdots,n-1\}.$$ Since $x \in supp(\mu_k)$ we conclude $x_j\in P_k$. Otherwise the neighbourhood $\mathcal{N}_j=\cdots\times [0,1] \times U_j\times [0,1] \times \cdots$ of $x$ has zero $\mu_k$-measure, where $U_j\subset [0,1]$ is an open set containing $x_j$ with $U_j \cap P_k=\emptyset$. From the choice of $r$ we conclude $x_j$ can take only one value in $P_k$, say $p_{i_j}$. 
\end{proof}

This lemma in particular implies that $$\mu_k(B_n(q,r))\le  \mu_k(A_{i_0,i_2,\dots, i_{n-1}})=\frac{1}{k^n},$$
for every $q\in [0,1]^\Z$. Then if $A$ satisfy $\mu_k(A)>1-\delta$ and $A\subset \bigcup_{i=1}^L B_n(z_i,r)$, then $$1-\delta<\mu_k(A)\le \mu_k(\bigcup_{i=1}^L B_n(z_i,r))\le \frac{L}{k^n}.$$
This implies that $N_{\mu_k}(n,r,\delta)\ge(1-\delta)k^n$, for every $n\ge 1$ and $\delta\in (0,1)$. Finally

$$\lim_{r \to 0}\frac{\sup_\mu h_{\mu}(r,T,\delta)}{|\log r|}\ge \lim_{k \to \infty}\frac{h_{\mu_k}(\frac{1}{3k},T,\delta)}{\log 3k}\ge \lim_{k \to \infty} \frac{\log k}{\log 3k}=1.$$
The opposite inequality is easier, we refer the reader to  \cite{LT} for the details. 

%from an identical calculation showed in . Let $\epsilon \in (0,1)$ and consider a natural number $l$ such that $\sum_{|j|\ge l} 2^{-|j|}<\epsilon /2$.  Let $M=\lfloor \frac{24}{\e}-1\rfloor$ and for each $k\in \{1,...,M\}$ define the interval $$I_k=[(k-1)\epsilon /24,(k+1)\epsilon /24),$$ and 
%$$I_{M+1}=((M-1)\epsilon /24,1].$$ Define the sets $$S_{i_{-l},\dots, i_{n+l}}=\{x\in [0,1]^\Z: x_j\in I_{i_j}\textrm{ for all } -l\le j \le n+l\},$$ where $i_j\in \{1,...,M+1\}$ for each $j$. Observe that $x,y \in S_{i_{-l},\dots, i_{n+l}}$ implies $$d(T^ix, T^iy)\le\sum_{|j|\ge l} 2^{-|j|}+\sum_{|j|\le l}2^{-|j|}d(x_{i+j},y_{i+j})<\frac{\epsilon}{2}+\frac{\epsilon}{12} \sum_{|j|\le l}2^{-|j|}\le \frac{\epsilon}{2}+\frac{\epsilon}{4}<\epsilon,$$
%for every $0\le i\le n$. This inequality proves that every set $S_{i_{-l},\dots, i_{n+l}}$ has diameter less than $\epsilon$ with respect to the metric $d_n$. Since the collection of sets $S_{i_{-l},\dots, i_{n+l}}$ is a covering of $[0,1]^\Z$ we conclude $N_{d_T}(n,\e)\le (M+1)^{n+2l+1}.$ Since for any measure $\mu$ we have $N_{\mu}(n,\e,\delta)\le N_{d_T}(n,\e)$,  we conclude
%$$\limsup_{\e \to 0}\frac{\sup_\mu h_{\mu}(\e,T,\delta)}{|\log \e|} \le \lim_{\epsilon \to 0}\frac{\lim_{n\to \infty} \frac{1}{n}\log N_{d_T}(n,\e)}{|\log \epsilon|}\le \lim_{\epsilon \to 0}\frac{\log \left(\frac{24}{\e} \right) }{|\log \epsilon|}=1.$$
\end{proof}

This computation trivially extends to higher dimensions. We recover the formula 
$$\overline{mdim}(([0,1]^d)^\Z,d_T,T)=d.$$
This can also be generalized to more general metric spaces. Let $Y$ be a compact metric space with metric $d$. Given $\e>0$, define $N(\e)$ as the maximal cardinality of an $\e$-separated set in $(Y,d)$. The upper Minkowski dimension or upper box dimension of $Y$ is defined as $$\overline{dim}_{B}(Y,d)=\limsup_{\e\to 0}\frac{N(\e)}{|\log\e|}.$$ Consider a decreasing sequence $\{\e_k\}_{k\in\N}$ converging to zero such that $\lim_{k\to \infty} \frac{N(\e_k)}{|\log\e_k|}= \overline{dim}_{B}(Y)$. The role of the points $P_k$ is replaced by a maximal collection of $\e_k$-separated points in $Y$. The measures $\lambda_k$ is the probability measure that equidistributes a maximal collection of $\e_k$-separated points. As before we get the inequality $\mu_k(B_n(q,r))<N(\e_k)^{-n}$, for every $q\in Y^\Z$ and $r<\frac{\e_k}{2}$.  We finally get the inequality 
$$\limsup_{\e \to 0}\frac{\sup_\mu h_{\mu}(r,T,\delta)}{|\log r|}\ge \limsup_{k \to \infty}\frac{h_{\mu_k}(\e_k/3,T,\delta)}{|\log(\e_k/3)|}\ge \lim_{k \to \infty} \frac{\log N(\e_k)}{|\log \e_k|}=\overline{dim}_{B}(Y).$$
We conclude $\overline{mdim}(Y^\Z,d_T,T)\ge \overline{dim}_{B}(Y).$ Let $\epsilon \in (0,1)$ and consider a natural number $l$ such that $\sum_{|j|\ge l} 2^{-|j|}<\epsilon /(2\cdot \diam (Y))$.  Let $M$ the maximum cardinality of an $\e$-separated subset $B=\{x_i\}_{i=1}^M$ of $\mathcal{X}$. Consider the function defined by $f(x)=x_i$, where $x_i$ is the closest point to $x$ in the subset $B$, whenever $x_i$ is uniquely defined. We can extend $f$ to a measurable function on $\X$. Define the sets $A_i=f^{-1}(x_i)$, and  $$S_{i_{-l},\dots, i_{n+l}}=\{y\in Y^\Z: y_j\in A_{i_j}\textrm{ for all } -l\le j \le n+l\},$$ where $i_j\in \{1,...,M\}$ for each $j$. Observe that $z,y \in S_{i_{-l},\dots, i_{n+l}}$ implies $$d(T^iz, T^iy)\le\diam(Y)\sum_{|j|\ge l} 2^{-|j|}+\sum_{|j|\le l}2^{-|j|}d(z_{i+j},y_{i+j})<\frac{\epsilon}{2}+2\epsilon \sum_{|j|\le l}2^{-|j|}<7\epsilon,$$
for every $0\le i\le n$. This inequality proves that every set $S_{i_{-l},\dots, i_{n+l}}$ has diameter less than $7\epsilon$ with respect to the metric $d_n$. Since the collection of sets $S_{i_{-l},\dots, i_{n+l}}$ is a covering of $Y^\Z$ we conclude $N_{d_T}(n,7\e)\le M^{n+2l+1}=N(\e)^{n+2l+1}.$ Since for any measure $\mu$ we have $N_{\mu}(n,\e,\delta)\le N_{d_T}(n,\e)$,  we conclude
\begin{align*}\limsup_{\e \to 0}\frac{\sup_\mu h_{\mu}(\e,T,\delta)}{|\log \e|} \le & \lim_{\epsilon \to 0}\frac{\limsup_{n\to \infty} \frac{1}{n}\log N_{d_T}(n,\e)}{|\log \epsilon|}\\ \le&  \limsup_{\epsilon \to 0}\frac{\log N(\epsilon)}{|\log \epsilon/7|}=\overline{dim}_B(Y,d).
\end{align*}

We summarize this in the following result. 
\begin{theorem} \label{calculo} With the notation above we have $$ \overline{mdim}(Y^\Z,d_T,T)=\overline{dim}_B(Y,d).$$
\end{theorem}

%We consider a maximal $\epsilon$-separated subset $B=\{x_i\}_{i=1}^M$ of $\mathcal{X}$. For each point $x\in \mathcal{X}$ we consider a ball of radius $\epsilon /2$. Asumming that the metric space $\mathcal{X}$ is Lindelof, there exists a countable subcover $\{B(a_i,\epsilon /2)\}_{i=1}^{\N}$ of $\mathcal{X}$. We are going to define inductively a function $f:\mathcal{X}\to B$. If $x \in B(a_1,\epsilon /2)$, $f(x)$ is equal to the closest point in $B$ to the ball, then if $B(a_2,\epsilon /2)$ is disjoint with $B(a_1,\epsilon /2)$ the definition is analogous. But, if they are not disjoint, for every $x\in B(a_2,\epsilon /2)\setminus B(a_1,\epsilon /2)$, $f(x)$ is equal to the closest point in $B$ to the ball $B(a_2,\epsilon /2)$. The definition of $f$ follows inductively in every ball $B(a_i,\epsilon /2)$.

%\begin{corollary}
%Let $T: [0,1]^{\Z} \to [0,1]^{\Z}$ be the shift and $\mu_k= (\lambda_k)^{\otimes \Z}$ is the product measure, where $\lambda_k$ is the measure supported on the equidistributed finite set $P:=\{p_1, p_2, \cdots, p_k\}$. Then $mdim_{\mu_k}( [0,1]^{\Z},T,d)=1.$
%\end{corollary}

\section{Rate distortion function and $\widetilde{h}_\mu(\e,T,\delta)$}  \label{sec4}

 \begin{proposition} \label{prop1}  Let $\mu$ be an ergodic $T$-invariant probability measure. Assume $(\X,d)$ is compact and denote its diameter by $D$. Then $$R_\mu(\epsilon)\le \widetilde{h}_\mu(\e,T, \e/2D). $$
Let $\p$ be any partition of diameter less than $\e$. Then
$$R_\mu(\epsilon)\le h(\p,T). $$

\end{proposition}

\begin{proof}

Recall that by definition $R_\mu(\epsilon)$ is the infimum of the quantities $\frac{1}{n}I(X,Y)$, where $(X,Y)$ satisfies condition $(*)_{n,\epsilon,\mu}$ (here $n$ is also allowed to vary). In particular if we want to give an upper bound to $R_\mu(\epsilon)$ we can just exhibit a good choice of pair $(X,Y)$. Since $(\X,d)$ is a compact metric space it has finite diameter, call $D$ to the diameter of $(\X,d)$. Choose $K\subset \X$ such that $\widetilde{N}(K,n,\epsilon)=\widetilde{N}_\mu(n,\epsilon, \frac{\epsilon}{2D})= N$, i.e. there exist points $\{x_i\}_{i=1}^N$ in $\X$ such that $K\subset \bigcup_{i=1}^N \widetilde{B}_n(x_i, \e/2D)$ and $\mu(K)>1-\frac{\epsilon}{2D}$, where last union is disjoint. Pick a point $p\in \X\setminus \{x_i\}_{i=1}^N$ and define a measurable function $f:\X\to\X$ in such way that for $x\in K$ we have that $f(x)=x_i$ for some $x_i$ with $x\in \widetilde{B}_n(x_i,\e/2)$, if $x\in \X\setminus K$, then $f(x)=p$. By construction  $|f(\X)|=N+1$. Define $Y_k=T^k\circ f\circ X$. Observe that 
\begin{align*}
& \E\Big(\dfrac{1}{n}\sum_{k=0}^{n-1}d(T^kX,Y_k)\Big)=\int_\X \dfrac{1}{n}\sum_{k=0}^{n-1}d(T^k(x),T^kf(x))d\mu(x) \\
 &= \int_K \dfrac{1}{n}\sum_{k=0}^{n-1}d(T^k(x),T^kf(x))d\mu(x)+\int_{\X\setminus K} \dfrac{1}{n}\sum_{k=0}^{n-1}d(T^k(x),T^kf(x))d\mu(x)\\
&\le \int_K \dfrac{1}{n}\sum_{k=0}^{n-1}d(T^k(x),T^kf(x))d\mu(x)+\mu(\X\setminus K) D .
\end{align*}
Finally observe that  by definition of $f$ and if $x\in K$ then we have $\widetilde{d}_n(x,f(x))\le \epsilon/2$. We immediately conclude that $$\E\Big(\dfrac{1}{n}\sum_{k=0}^{n-1}d(T^kX,Y_k)\Big)\le \e.$$
Since by definition of $I(X,Y)$ we have $I(X,Y)\le \log|f(\X)|=\log(N+1)$, then 
$$R_\mu(\e)=\inf \dfrac{1}{n}I(X,Y)\le \liminf_{n\to\infty}\dfrac{1}{n}\log \widetilde{N}(n,\e,\e/(2D)).$$
If we construct $f$ using sets well approximated by $(n,\e)$-dynamical balls, i.e. in the construction above  we take $K$ to satisfy $N(K,n,\frac{\e}{2D})=N_\mu(n,\e,\frac{\e}{2D})$, we obtain the analogous bound  
$$R_\mu(\e)\le \liminf_{n\to\infty}\dfrac{1}{n}\log N(n,\e,\e/(2D)).$$
Finally use Remark \ref{rem2} to conclude.
\end{proof}

The following theorem is an equivalent version of the Source Coding Theorem, which is a fundamental result in Information theory. We will use Theorem \ref{rd} in the proof of Proposition \ref{prop2}. For a proof and history about this theorem we refer the reader to \cite{G}. 

\begin{theorem}\label{rd} Let $X:(\Omega,\P)\to (\X,\mu)$ be a measurable function with $X_*\P=\mu$, an ergodic $T$-invariant probability measure on $\X$. Given $\e_0>0$, there exists $n(\e_0)\in \N$ such that the following holds. For each $n\ge n(\e_0)$, there exists a measurable function $f_n:\X\to \X^n$  such that the pair $(X,f_n\circ X)$ satisfies condition $(*)_{n,\e+\e_0}$   and $|f_n(\X)|\le e^{n(R_\mu(\e)+\e_0)}$.
\end{theorem}

The following definition has similar content than Definition \ref{cond} but will help to simplify the language used in the proof of Proposition \ref{prop2}.
\begin{definition} A map $Z$ is called a $(n,\epsilon)$-approximation of $(\X,T)$ if $Z:\X\to\X^n$ is a measurable map with finite image  and  it satisfies $$\int_\X \frac{1}{n}\sum_{k=0}^{n-1} d(T^k(x), Z_k(x))d\mu(x)\le \epsilon.$$ 
\end{definition}

\begin{lemma}\label{lem} Given $Z$ a $(n,\epsilon)$-approximation of $(\X,T)$, we can find a $(n,2\epsilon)$-approximation of $(\X,T)$, say  $Z'=(Z'_0,...,Z'_{n-1})$, such that $Z'_k=T^kZ'_0$. Moreover the preimage partitions of $Z$ and $Z'$ coincide. 
\end{lemma}

\begin{proof} Denote by $\Q=\{Q_1,...,Q_T\}$ the preimage partition of $Z$, this means that $Z_{|Q_j}\equiv z_j=(z_{j,0},...,z_{j,n-1})\in \X^n$. Then we have $$\sum_{k=1}^T\int_{Q_k}\frac{1}{n}\sum_{i=0}^{n-1}d(T^i(x), z_{k,i})d\mu(x)\le \epsilon.$$
To simplify notation define $g_k:Q_k\to \R$ as $g_k(x)=\frac{1}{n}\sum_{i=0}^{n-1}d(T^i(x), z_{k,i})$. Therefore the above inequality can be writen as $$\sum_{k=1}^T \int_{Q_k}g_k(x)d\mu(x)\le \epsilon.$$ 
Let $a_k:= \int_{Q_k}g_k(x)d\mu(x)$. Assume $a_k\ne 0$ for some $k$, we claim that there exists $x_k\in Q_k$ such that $g_k(x_k)\le a_k/\mu(Q_k)$. If $g_k(x)> a_k/\mu(Q_k)$ for all $x\in Q_k$, then integrating over $Q_k$ will lead to a contradiction by the definition of $a_k$. If $a_k=0$, we pick any element in $Q_k$  and call it $x_k$. We define $Z'_i(x)=T^i(x_k)$ whenever $x\in Q_k$. By construction the preimage partions of $Z$ and $Z'$ coincide and 
\begin{align*}\int_{Q_k}\frac{1}{n}\sum_{i=0}^{n-1}d(T^i(x), T^i(x_k))d\mu(x)&\le \int_{Q_k}\frac{1}{n}\sum_{i=0}^{n-1}d(T^i(x), z_{k,i})d\mu(x)+\mu(Q_k)g_k(x_k)\\
& \le 2\epsilon.
\end{align*}

\end{proof}

\begin{proposition}\label{prop2} Let $\mu$ be an ergodic $T$-invariant probability measure. The following inequality holds 

$$\widetilde{h}_\mu(4L\epsilon, T,1/L)\le  R_\mu(\epsilon).$$

\end{proposition}

\begin{proof} 

%Let $Z:\X\to\X^n$ be a measurable map with finite image. Denote by $\{p_l\}_{l=1}^L\subset \X^n$  the image of the map $Z$ and denote by $\p_i$ the preimage of $p_i$ under $Z$. The set $\{\p_i\}_{i=1}^L$ is a finite measurable partition of $\X$. Since $\X$ is a standard space we can find a sequence of finite partitions $\p^{(n)}$, where $\p^{(0)}=\p$, $\p^{(n)}\subset \p^{(n+1)}$ and the $\sigma$-algebra generated by $\bigcup_n \p^{(n)}$ is the Borel  $\sigma$-algebra on $X$. Choose maps $q^{(n)}:\X\to \N$ such that the preimage partition of $q^{(n)}$ coincides with  $\p^{(n)}$, in particular $q^{(n)}$ has finite image. 
%Then we have %$$I(X,Z\circ X)=\lim_{n\to\infty} I(q^{(n)}\circ X, Z\circ X)=\lim_{n\to\infty} H(q^{(n)}(X))+H(Z\circ X)-H(Z\circ X,q^{(n)}(X)).$$
%By definition $H(Z\circ X,q^{(n)}\circ X)$ is the entropy of the refinement of the preimage partitions associated to $Z\circ X$ and $q^{(n)}\circ X$, but by construction of $q^{(n)}$ this is just the preimage partition associated to $q^{(n)}\circ X$. In other words $ H(q^{(n)}(X))=H(Z\circ X,q^{(n)}(X))$, this implies that 
%$$I(X,Z\circ X)=H(Z\circ X).$$
%Observe that since $X_*\P=\mu$ then $H_\P(Z\circ X)=H_\mu(Z)=H_\mu(\p)$. 

%We can now rephrase the problem of minimizing the expression $\frac{1}{n}I(X,Y)$ under the constraint $\E(\frac{1}{n}\sum_{k=1}^{n-1} d(T^kX,Y_k))\le \epsilon$ to the following more concrete problem: \\

%\noindent
%{\bf Goal:} Minimize $\frac{1}{n}H(\p)$ when $\p$ is the preimage partition of some $(n,\epsilon)$-approximation of $(\X,T)$.\\

Fix $\e_0>0$ arbitrarily small. For $n\ge n(\e)$ consider the map  $f_{n}=(f_{n,0},...,f_{n,n-1}):\X\to \X^n$ provided by Theorem \ref{rd}. Since $(X,f_n\circ X)$ satisfies condition $(*)_{n,\epsilon+\e_0}$ we have the inequality $$\int_\X \frac{1}{n}\sum_{k=0}^{n-1}d(T^k(x),f_{n,k}(x))d\mu(x)\le \epsilon+\e_0.$$
In particular each map $f_n$ is a $(n,\epsilon+\e_0)$-approximation of $(\X,T)$. Moreover 
\begin{align*} \frac{1}{n}\log|f_n(\X)|\le & R_\mu(\epsilon)+\e_0.
%\le & \dfrac{1}{n} I(X,f_n\circ X)+\epsilon= H_\mu(f_n) +\delta \\
%\le & \frac{1}{n}\log|f_n(\X)| +\delta.
\end{align*}

%In particular $$H_\mu(f_n)\ge \frac{1}{n}\log|f_n(\X)| -\delta\ge R_\mu(\epsilon)-2\delta$$ and $$H_\mu(f_n)\le \frac{1}{n} \log|f_n(\X)|\le R_\mu(\epsilon)+\delta$$
Using Lemma \ref{lem} we find maps $\widetilde{f}_n:\X\to\X^n$, each being a $(n,2\epsilon+2\e_0)$-approximation of $(\X,T)$ with the additional property that $\widetilde{f}_{n,k}=T^k\widetilde{f}_{n,0}$. If we define $g_n:=\widetilde{f}_{n,0}$, then $\widetilde{f}_n=(g_n,Tg_n,...,T^{n-1}g_n)$. Observe that $H_\mu(f_n)=H_\mu(\widetilde{f}_n)$ and $|f_n(\X)|=|\widetilde{f}_n(\X)|=|g_n(\X)|$, in particular $$ \frac{1}{n}\log|g_n(\X)|\le  R_\mu(\epsilon)+\e_0.$$
We have now all the ingredients to relate $R_\mu(\epsilon)$ with the quantity $\widetilde{h}_\mu(2L\e,T,1/L)$. First notice that 
$$\int_\X \widetilde{d}_n(x,g_n(x) ) d\mu(x) = \int_\X \dfrac{1}{n}\sum_{k=0}^{n-1} d(T^k(x),T^kg_n(x))d\mu(x)\le 2\epsilon+2\e_0.$$
Define $\c=g_n(\X)\subset \X$ and set $A=\bigcup_{p\in \c} \widetilde{B}_n(p,2L(\epsilon+\e_0))$. Observe that if $x\in \X\setminus A$ then $\widetilde{d}_n(x,g_n(x))\ge 2L(\epsilon+\e_0)$, this immediately implies $\mu(\X\setminus A)\le 1/L$ or equivalently $\mu(A)\ge 1-\frac{1}{L}$. By construction the set $A$ can be covered by $|\c|=|g_n(\X)|$ $(n,2L(\epsilon+\e_0))$-average dynamical balls. This implies $$\dfrac{1}{n}\log \widetilde{N}_\mu(n,2L(\epsilon+\e_0),1/L)\le \dfrac{1}{n}\log |g_n(\X)|\le R_\mu(\epsilon)+\e_0,$$
for each $n>n(\e_0)$. This implies 
$$\limsup_{n\to \infty} \dfrac{1}{n}\log \widetilde{N}_\mu(n,2L(\epsilon+\e_0),1/L)\le  R_\mu(\epsilon)+\e_0,$$
for each $\e_0>0$. Taking $\e_0\to 0$,
$$\limsup_{n\to \infty} \dfrac{1}{n}\log \widetilde{N}_\mu(n,4L\e,1/L)\le \limsup_{n\to \infty} \dfrac{1}{n}\log \widetilde{N}_\mu(n,2L(\epsilon+\e_0),1/L)\le   R_\mu(\epsilon)+\e_0,$$
we get
$$\widetilde{h}_\mu(4L\e,T,1/L)=\limsup_{n\to \infty} \dfrac{1}{n}\log \widetilde{N}_\mu(n,4L\e,1/L)\le R_\mu(\epsilon).$$
\end{proof}

We can summarize Proposition \ref{prop1} and Proposition \ref{prop2} in the following result.

\begin{theorem} \label{teo1} Let $(\X,d)$ be a compact metric space of diameter $D$ and $T:\X\to \X$ a continous transformation. Let $\mu$ be an ergodic $T$-invariant probability measure, then we have 
 $$\widetilde{h}_\mu(4L\epsilon,T,1/L)\le R_\mu(\epsilon)\le \widetilde{h}_\mu(\epsilon,T,\e/2D).$$
\end{theorem}

%\begin{remark} The lower bound in Theorem \ref{teo1} does not require the compactness of $\X$. In the noncompact case we need to ensure that the measurable space $\X$ is still standard (to use Theorem \ref{rd}). If $\X$ is metrizable, locally compact and second countable, then it admits a compact exhaustion. This is  enough to ensure that the borelian $\sigma$-algebra is standard. In particular for any smooth manifold with a metric generating the topology we have this lower bound. 
%\end{remark}

Theorem \ref{teo1} will be used in the proof of Theorem \ref{theo2}. 

\begin{definition} Given $x\in\X$, $n\ge0$ and $r\in (0,1)$ we define
$$B'_n(x,\e,r)=\{y\in X: \dfrac{1}{n}\#\{0\le i\le n-1:d(T^ix,T^iy)\le \e\}\ge 1-r\}.$$
Define $N_\mu(n,\e,\delta,r)$ as the minimum number of balls $B'_n(x,\e,r)$ needed to cover a set of measure bigger or equal than $\delta$.
\end{definition}
The following proposition follows from the main result of \cite{ZZC}. 

\begin{proposition} \label{prop3} Let $\mu$ be an ergodic $T$-invariant probability measure. Then for each $\delta\in (0,1)$ we have the inequality
$$h_\mu(T)\le \lim_{r\to 0}\lim_{\e\to 0}\limsup_{n\to \infty} \dfrac{1}{n}\log N_\mu(n,\e,\delta,r).$$
If we moreover assume that $\X$ is compact, then the equality holds.
\end{proposition}

An easy application of Proposition \ref{prop3} is the following result.

\begin{proposition} \label{prop4} Let  $\mu$  be an ergodic $T$-invariant probability measure. Then  for every $\delta\in (0,1)$ we have
$$h_\mu(T)\le \widetilde{h}_\mu(T,\delta).$$
\end{proposition}
 
\begin{proof}
Observe that if $y\in \widetilde{B}_n(x,\e)$, then $$\#\{0\le i\le n-1: d(T^ix,T^iy)>L\e\}\le n/L, $$ or equivalently $$\dfrac{1}{n}\#\{0\le i\le n-1: d(T^ix,T^iy)\le L\e\}\ge 1-\dfrac{1}{L} .$$ This implies that $ \widetilde{B}_n(x,\e)\subset B_n'(x,L\e,1/L)$, and therefore $$N_\mu(n,L\e,\delta,1/L)\le \widetilde{N}_\mu(n,\e,\delta).$$
By taking limits we get 
\begin{align*}
\lim_{\e\to 0}\limsup_{n\to\infty} \dfrac{1}{n}\log N_\mu(n,\e,\delta,r) \le& \lim_{\e\to 0}\limsup_{n\to\infty}\dfrac{1}{n}\log \widetilde{N}_\mu(n,\e,\delta)=\widetilde{h}_\mu(T,\delta).
\end{align*}
Finally using Proposition \ref{prop3} we get $$h_\mu(T)\le \widetilde{h}_\mu(T,\delta).$$
\end{proof}

We can finally prove Theorem \ref{theo2}.

\begin{corollary} Under the hypothesis of Theorem \ref{teo1} we have 
 $$\widetilde{h}_\mu(T,\delta)= \lim_{\epsilon\to 0} R_\mu(\epsilon)= h_\mu(T)$$
%If $(\X,T)$ is $n$-expanding (see Definition \ref{}), then $$h_\mu(T)=\lim_{\e\to 0}R_\mu(\epsilon).$$
\end{corollary}
\begin{proof} Combining Theorem \ref{teo1} and Proposition \ref{prop1} we get   $$\widetilde{h}_\mu(T,\delta)\le \lim_{\epsilon\to 0} R_\mu(\epsilon)\le  h_\mu(T).$$
Then Proposition \ref{prop4} gives the equality of the three quantities.
\end{proof}

We will now prove a version of the variational principle for the average dynamical distances. The proof follows closely the proof of Lemma \ref{1111} with minor modifications. For completeness we explain it in detail. Define $\widetilde{N}_d(n,\e)$ as the maximal cardinality of a $(n,\e)$-average separated set in $(\X,d)$ and $$\widetilde{S}(\X,d,\e)=\limsup_{n\to \infty}\dfrac{1}{n}\log \widetilde{N}_d(n,\e).$$ 
In case the dynamical system has been specified, we will frequently use the simplified notation $\widetilde{S}(\e)=\widetilde{S}(\X,d,\e)$. 

\begin{lemma} \label{2222} Assume $(\X,d)$ is compact and $T$ is injective. Given $\e>0$ and $\delta\in (0,1/4)$, there exists a $T$-invariant probability measure $\mu_\e$ such that 
$$ \widetilde{h}_{\mu_\e}(\e,T,\delta)\ge  \widetilde{S}\Big(\dfrac{2\e}{1-4\delta}\Big)-3\delta  \widetilde{S}\Big(\dfrac{\e}{1-4\delta}\Big).$$
\end{lemma}

\begin{proof} We will fix $\e>0$ and $\delta\in (0,1/4)$. Let  $E_n=\{x_1,...,x_{\widetilde{N}_d(n,\e)}\}$ be a maximal collection of $(n,\e)$-average separated points in $\X$.  Define $$\sigma_n=\dfrac{1}{|E_n|}\sum_{x\in E_n} \delta_x,$$
where $\delta_x$ is the probability measure supported at $x$.  Then define   
$$\overline{\sigma}_n=\dfrac{1}{n}\sum_{k=0}^{n-1} T^k_*\sigma_n.$$
Consider a subsequence $\{n_k\}$ such that $$\widetilde{S}(\X,d,\e)=\lim_{k\to \infty}\dfrac{1}{n_k}\log \widetilde{N}_d(n_k,\e).$$ By standard arguments we can find a subsequence of $\{n_k\}$, that we still denote by $\{n_k\}$, such that $\{\overline{\sigma}_{n_k}\}_{k\in\N}$ converges to a $T$-invariant probability measure $\mu$. We  arrange the sequence such that $\lim_{k\to \infty} \frac{n_k}{k}=\infty$. Let $K$ be a subset of $\X$ with $\mu(K)> 1-\delta$ and $\widetilde{N}_d(K,n,(1-4\delta)\e/2)=\widetilde{N}_\mu(n,(1-4\delta)\e/2,\delta)$ where $\widetilde{N}(K,m,r)$ is defined to be the minimum number of $(m,r)$-average dynamical balls needed to cover $K$. We can assume that $K$ is open in $\X$. There exists $k_0$ such that for every $k\ge k_0$ we have $\overline{\sigma}_{n_k}(K)>  1-\delta$. Let $$L_n=\{(i,j)\in\N^2: 0\le i\le n-1, 1\le j\le \widetilde{N}(n,\e)\}.$$ We assign to each point in $L_n$ either a $0$ or $1$ in the following way. If $T^ix_j\in K$ then assign $1$ to the point $(i,j)$,  assign $0$ to $(i,j)$ otherwise. By definition of the measure $\overline{\sigma}_{n_k}$ we know that the number of ones in $L_{n_k}$ is bigger or equal than $n_k \widetilde{N}_d(n_k,\e)(1-\delta)$. For $s>1$ we define $L_{n_k}(s)$ as the set of points in $L_{n_k}$ with first coordinate in the interval $[[n_k/k],n_k-[n_k/s]-1]$. The number of ones in $L_{n_k}(s)$ is at least $n_k \widetilde{N}_d(n_k,\e)(1-\delta-\frac{1}{n_k}[\frac{n_k}{s}]-\frac{1}{n_k}[\frac{n_k}{k}])\ge n_k \widetilde{N}_d(n_k,\e)(1-\delta-\frac{1}{s}-\frac{1}{k})$. From now on we assume $s>\frac{1}{1-2\delta-\frac{1}{k}}$, in particular $1-\delta-\frac{1}{s}-\frac{1}{k}>\delta$. We will moreover assume $s<\frac{1}{1-3\delta}$. This can be done if we assume $k$ is sufficiently large so that $\delta>1/k$. We conclude that the number of ones in $L_{n_k}(s)$ is at least  $n_k \widetilde{N}_d(n_k,\e)\delta$. Since $L_{n_k}$ has $(n_k-[n_k/s]-[n_k/k])$ columns, then by the pigeonhole principle there exists an index $m_k$ such that the $m_k$th column has at least $\frac{n_k \widetilde{N}_d(n_k,\e)\delta}{(n_k-[n_k/s]-[n_k/k])}$ ones and $[n_k/k]\le m_k< n_k-[n_k/s]$. Observe that $\X$ can be covered by at most $\widetilde{N}_d(m,\e/2)$ $(m,\e/2)$-average dynamical balls, in particular with $\widetilde{N}_d(m_k,\e/2)$ subsets of $d_{m_k}$-diameter smaller than $\e$. Also observe that if $i\ne j$ and $\widetilde{d}_{m_k}(x_i,x_j)\le \e$, then 
\begin{align*}
m_k\e+\sum_{p=0}^{n_k-m_k-1} d(T^{m_k+p}x_i,T^{m_k+p}x_j)\ge& \sum_{p=0}^{m_k-1}d(T^{p} x_i,T^{p}x_j)+\sum_{p=m_k}^{n_k-1} d(T^px_i,T^px_j)\\ &> n_k\e. 
\end{align*}
In the last inequality, we are using the fact that $\widetilde{d}_{n_k}(x_i,x_j)>\e$.  This implies that $\widetilde{d}_{n_k-m_k}(T^{m_k}x_i,T^{m_k}x_j)> \e$. We can conclude that there exists an subset $I\subset \{1,....,\widetilde{N}_d(n_k,\e)\}$ such that for $i\in I$ we have $T^{m_k}x_i\in K$, and $|I|\ge n_k \widetilde{N}_d(n_k,\e)\delta/(n_k-[n_k/s]-[n_k/k])>\widetilde{N}_d(n_k,\e)\delta$. By the pigeonhole principle there exists a subset $A$ of $\{T^{m_k} x_i\}_{i\in I}$ such that  the diameter of $A$ with respect to $\widetilde{d}_{m_k}$ is at most $\e$ and $$|A|\ge \dfrac{ \widetilde{N}_d(n_k,\e)\delta}{\widetilde{N}_d(m_k,\e/2)  }.$$ As mentioned above this implies that if $a,b\in A$ and $a\ne b$, then $\widetilde{d}_{n_k-m_k}(a,b)\ge \e$. Therefore  $$\widetilde{d}_{n_k}(a,b)>\frac{n_k-m_k}{n_k}\e>\frac{\frac{n_k}{s}-1}{n_k}\e>\Big(\dfrac{1}{s}-\dfrac{1}{n_k}\Big)\e>\Big(1-3\delta-\frac{1}{n_k}\Big)\e>(1-4\delta)\e.$$
We have assumed $k$ is sufficiently large so that $\delta>\frac{1}{n_k}$. Finally we obtain the bound
$$\widetilde{N}_\mu(n_k,(1-4\delta)\e/2,\delta)=\widetilde{N}_d(K,n_k,(1-4\delta)\e/2)\ge  P \ge \dfrac{ \widetilde{N}_d(n_k,\e)\delta}{\widetilde{N}_d(m_k,\e/2)  },$$
where $P$ is the maximum number of $(n,(1-4\delta)\e)$-average separated points in $K$. Finally
\begin{align*}\widetilde{h}_\mu((1-4\delta)\e/2,T,\delta)=&\limsup_{n\to \infty}\dfrac{1}{n}\log \widetilde{N}_\mu(n,(1-4\delta)\e/2,\delta)\\ \ge & \limsup_{k\to\infty} \dfrac{1}{n_k}\log \widetilde{N}_d(n_k,\e)-\dfrac{m_k}{n_k}\dfrac{1}{m_k}\log \widetilde{N}_d(m_k,\e/2).
\end{align*}
Recall that by construction $m_k<n_k-[n_k/s]$, then $\frac{m_k}{n_k}\le 1-\frac{1}{s}+\frac{1}{n_k}$. Therefore
\begin{align*} \widetilde{h}_\mu((1-4\delta)\e/2,T,\delta)&\ge \limsup_{k\to\infty} \dfrac{1}{n_k}\log \widetilde{N}_d(n_k,\e)-(1-\dfrac{1}{s}+\dfrac{1}{n_k})\dfrac{1}{m_k}\log \widetilde{N}_d(m_k,\e/2).\\
&\ge \limsup_{k\to\infty} \dfrac{1}{n_k}\log \widetilde{N}_d(n_k,\e)-(3\delta+\dfrac{1}{n_k})\dfrac{1}{m_k}\log \widetilde{N}_d(m_k,\e/2).
\end{align*}
 Also observe that the inequality $[n_k/k]\le m_k$ implies that $m_k\to \infty$ as $k\to \infty$, and by construction $\widetilde{S}(\X,d,\e)=\lim_{k\to \infty}\dfrac{1}{n_k}\log \widetilde{N}_d(n_k,\e)$, therefore $$ \widetilde{h}_\mu((1-4\delta)\e/2,T,\delta)\ge  \widetilde{S}(\e)-3\delta \widetilde{S}(\e/2).$$
This implies in particular that $$\sup_{\mu\in\M_T(\X)} \widetilde{h}_{\mu}(\e,T,\delta)\ge  \widetilde{S}\Big(\dfrac{2\e}{1-4\delta}\Big)-3\delta  \widetilde{S}\Big(\dfrac{\e}{1-4\delta}\Big).$$

%\begin{align*}\widetilde{h}_\mu((1-4\delta)\e/2,T,\delta/2)=&\limsup_{n\to \infty}\dfrac{1}{n}\log \widetilde{N}_\mu(n,(1-4\delta)\e/2,\delta/2)\\ \ge & \limsup_{k\to\infty} \dfrac{1}{n_k-m_k}\log \widetilde{N}(n_k,\e)-\dfrac{m_k}{n_k-m_k}\dfrac{1}{m_k}\log \widetilde{N}(m_k,\e/2)\\
%&\ge \limsup_{k\to\infty} \dfrac{1}{n_k}\log \widetilde{N}(n_k,\e)-\dfrac{m_k}{n_k-m_k}\dfrac{1}{m_k}\log \widetilde{N}(m_k,\e/2).
%\end{align*}
%Recall that by our choices we know  $m_k<n_k-[n_k/s]$ and $\frac{1}{s}>1-3\delta$. This together implies 
%$$\dfrac{m_k}{n_k-m_k}\le \dfrac{n_k-\frac{n_k}{s}+1}{\frac{n_k}{s}-1}=\dfrac{1-\frac{1}{s}+\frac{1}{n_k}}{\frac{1}{s}-\frac{1}{n_k}}\le \dfrac{3\delta+\frac{1}{n_k}}{1-3\delta-\frac{1}{n_k}}.$$
%Then
%\begin{align*}\widetilde{ h}_\mu(\e/2,T,\delta)&\ge \limsup_{k\to\infty} \dfrac{1}{n_k}\log \widetilde{N}(n_k,\e)-(\dfrac{3\delta+\frac{1}{n_k}}{1-3\delta-\frac{1}{n_k}})\dfrac{1}{m_k}\log \widetilde{N}(m_k,\e/2).
%\end{align*}

%Also observe that the inequality $[n_k/k]\le m_k$ implies that $m_k\to \infty$ as $k\to \infty$.
%Then 
%$$\widetilde{h}_\mu(\e/2,T,\delta)\ge  \widetilde{S}(\e)-\dfrac{3\delta}{1-2\delta}  \widetilde{S}(\e/2).$$
%This implies in particular that $$\sup_\mu \widetilde{h}_\mu(\e,T,\delta)\ge \widetilde{S}(2\e)-\dfrac{3\delta}{1-2\delta} \widetilde{S}(\e),$$
%where the supremum runs over the set of $T$-invariant probability measures.
\end{proof}

As corollary of Lemma \ref{2222} we will obtain a proof of Theorem \ref{teolin}. For this we need a version of Theorem \ref{rd} for non-ergodic measures. To state this result we need to introduce some notation. On $\X^n$ we consider the  metric $\rho_n$ given by $\rho_n(x,y)=\frac{1}{n}\sum_{i=1}^n d(x_i,y_i)$, where $x=(x_1,...,x_n)$ and $y=(y_1,...,y_n)$. We define $\mu^n$ as the image of $\mu$ under the map $\X\to \X^n$ given by $x\mapsto (x,Tx,...,T^{n-1}x)$. For a finite subset $C_n\subset \X^n$ we define $\E_\mu(C_n)$ as the integral $\int _{\X^n}\min_{c\in C_n} \rho_n(x,c)d\mu^n(x)$. 
\begin{definition} Given $R>0$ define $$\delta_\mu(R)=\inf\{\E_\mu(C_n)|  \text{ }C_n\subset \X^n \text{ and }|C_n|\le e^{nR}\},$$
and
$$D_\mu(R)=\inf\{\E_\mu\Big(\frac{1}{n}\sum_{k=0}^{n-1}d(T^kX,Y_k\circ X)\Big)|  \text{ }\frac{1}{n}I_\mu(X,(Y_0,...,Y_{n-1}))\le R\}.$$
\end{definition}
In this language the Source Coding Theorem (see Theorem \ref{rd}) can be stated as 
\begin{theorem} Let $\mu$ be an ergodic $T$-invariant probability measure on $\X$. Then $$\delta_\mu(R)=D_\mu(R).$$
\end{theorem}
The following properties of $\delta_\mu(R)$ are proven in \cite{G}.
\begin{proposition} The function $\mu\mapsto \delta_\mu(R)$ is affine and upper semicontinuous. The function $R\mapsto \delta_\mu(R)$ is convex and decreasing. 
\end{proposition}
As a corollary we obtain the formula $$\delta_\mu(R)=\int_\X \delta_{\mu_x}(R)d\mu(x),$$
where $\mu=\int \mu_x d\mu(x)$ is the ergodic decomposition of $\mu$. It is easy to see that $\inf_{R\ge 0} \delta_\mu(R)=0$. This together with the convexity of $\delta_\mu(R)$ implies that $\delta_\mu(R)$ is strictly decreasing at $R$ if $\delta_\mu(R)>0$. Define $$D_\infty(R)=\sup_x D_{\mu_x}(R)=\sup_x \delta_{\mu_x}(R).$$
Since the supremum of convex function is still convex we conclude $D_\infty(R)$ is convex, in particular continuous. It follows easily from the definition that $\inf_{R\ge 0} D_\infty(R)=0$. As before this implies that $D_\infty(R)$ is strictly decreasing at $R$ if $D_\infty(R)>0$. We remark that $R_\mu(\e)$ is the `formal' inverse of $D_\mu(R)$, with this we mean it is the inverse in the region where the functions are strictly decreasing. For the next remark we will use the following fact.

\begin{lemma} \label{lemconv}
Let $D_{\mu_x}: (0,\infty) \to \R$ be a convex decreasing function for each $x\in\X$ and define $$D_{\infty}(R)=\sup_x D_{\mu_x}(R).$$ Let $I=\{t\in (0,\infty): D_\infty(t)>0\}$ and $J=Im(D_\infty)$. Then $D_\infty^{-1}:J\to I$ is well defined and for every $\e\in J$ we have
$$D_{\infty}^{-1}(\epsilon)=\sup_{x} D_{\mu_x}^{-1}(\epsilon).$$ 
\end{lemma}
\begin{proof} As mentioned above the convexity of $D_{\mu_x}$ implies the convexity and therefore the continuity of $D_\infty$. We also know that $D_\infty$ is strictly decreasing at $R$ if $D_\infty(R)>0$. This implies that $D_\infty^{-1}:J\to I$ is well defined and $J$ is a connected open interval. The inequality  $D_{\infty}^{-1}(\epsilon)\ge \sup_{x} D_{\mu_x}^{-1}(\epsilon)$ follows from the decreasing assumption. For the other inequality we argue by contradiction, i.e. we assume there exists $\e\in J$ and $\epsilon_0>0$ such that for all $x\in\X$ we have $D_{\infty}^{-1}(\epsilon)-\epsilon_0 >D_{\mu_x}^{-1}(\epsilon)$, whenever $D_{\mu_x}^{-1}(\epsilon)$ is well defined. Define $b=D_{\infty}^{-1}(\epsilon)$ and $a=D_{\infty}^{-1}(\epsilon)-\epsilon_0/2$. By the definition of $D_\infty$ there exists $x$ such that $D_{\infty}(a)\ge D_{\mu_x}(a) >\epsilon=D_{\infty}(b)\ge D_{\mu_x}(b)$. By the intermediate value theorem exists $c\in [a,b)$ such that $D_{\mu_x}(c)=\epsilon.$ This gives a contradiction since $c> D_{\infty}^{-1}(\epsilon)-\epsilon_0$.
\end{proof}

\begin{remark} \label{remnoerg} Given $R>0$ and $\e_0>0$, there exists $n(\e_0,R)\in \N$ such that the following holds. For $n\ge n(\e_0,R)$ there exists $C_n\subset \X^n$ such that 
$|C_n|\le e^{n(R+\e_0)}$ and $\E_\mu(C_n)\le \int D_{\mu_x}(R)d\mu(x) +\e_0$. 
In particular $\E_\mu(C_n)\le D_\infty(R) +\e_0$. If $R_\infty(\e)$ is the inverse of $D_\infty(R)$, then by Lemma \ref{lemconv} we know $R_\infty(\e)=\sup_{\mu_x} R_{\mu_x}(\e)$. In particular for $R=R_\infty(\e)$ and $n\ge n(\e_0,R)$ we get a code $C_n\subset \X^n$ such that $|C_n|\le e^{n(R_\infty(\e)+\e_0)}$ and $\E_\mu(C_n)\le \e+\e_0$. Following the proof of Proposition \ref{prop2} we get $$\widetilde{h}_\mu(4L\e,T,1/L)\le \sup_x R_{\mu_x}(\e)\le\sup_{\mu\text{ ergodic}} R_{\mu}(\e).$$
\end{remark}

Before reproving Theorem \ref{teolin} we are going to state  Condition 1.2 introduced in \cite{LT} and mentioned in the statement of Theorem \ref{teolin}.

\begin{definition} \label{cond}
Let $(\mathcal{X},d)$ be a compact metric space. It satisfy the Condition 1.2, if for every $\delta>0$, we have $$\lim_{\epsilon\to 0} \epsilon^{\delta} \log \#(\mathcal{X},d,\epsilon)=0,$$
where $\#(\mathcal{X},d,\epsilon)$ is the minimal number of $\e$-balls needed to cover $\X$.
\end{definition}

\begin{remark}\label{remfin}
It is important to mention that every compact metrizable space admits a distance satisfying the Condition 1.2. This remark corresponds to Lemma 1.3 in \cite{LT}. It is also important to mention that Condition 1.2 implies that $$\limsup_{\e\to 0} \dfrac{\widetilde{S}(\X,d,\e)}{|\log\e|}=\overline{mdim}(\X,d,T).$$ 
\end{remark}

\begin{proposition} \label{proprate} Let $(\X,d)$ be a compact metric space and $T:\X\to \X$ a continuous map. Then 
$$\limsup_{\e\to 0} \dfrac{\sup_\mu R_\mu(\e)}{|\log \e|}= \limsup_{\e\to 0} \dfrac{\widetilde{S}(\X,d,\e)}{|\log\e|}.$$
If $(\X,d)$ satisfy Condition 1.2 in \cite{LT}, then $$\overline{mdim}(\X,d,T)=\limsup_{\e\to 0} \dfrac{\sup_\mu R_\mu(\e)}{|\log \e|}.$$
\end{proposition}

\begin{proof}  The inequality $$\limsup_{\e\to 0} \dfrac{\sup_\mu R_\mu(\e)}{|\log \e|}\le  \limsup_{\e\to 0} \dfrac{\widetilde{S}(\X,d,\e)}{|\log\e|},$$
is the easy part of the statement. We refer the reader to \cite{LT} for a proof. We will prove here the reversed inequality. Observe that by Remark \ref{remnoerg} we have
$$\widetilde{h}_\mu(4L\epsilon,T,1/L)\le \sup_{\mu\text{ ergodic} }R_\mu(\epsilon).$$
In particular $$\widetilde{h}_\mu(4\e/\delta,T,\delta)\le \sup_{\mu\text{ ergodic} }R_\mu(\epsilon).$$ This implies 
$$\sup_{\mu \text{ ergodic}} R_\mu(\e)\ge \sup_\mu \widetilde{h}_\mu(4\e/\delta,T,\delta)\ge \widetilde{S}\Big(\frac{8\e}{\delta(1-4\delta)}\Big)-3\delta\widetilde{S}\Big(\frac{4\e}{\delta(1-4\delta)}\Big).$$
Finally dividing by $|\log \e|$, taking limsup in $\e$ and then $\delta \to 0$ we obtain the lower bound
$$\limsup_{\e\to 0} \dfrac{\sup_{\mu \text{ ergodic}}R_\mu(\e)}{|\log \e|}\ge \limsup_{\e\to 0} \dfrac{\widetilde{S}(\e)}{|\log\e|}.$$ 

\end{proof}
Combining  Proposition \ref{proprate} and Remark \ref{remfin} we obtain a proof of Theorem \ref{teolin}.

\section{Dynamics on manifolds}\label{sec5}

We start this section with the following observation. 
\begin{remark} \label{remmm} Let $F:\X\to \X$ be a dynamical system. The map $x\mapsto (x,Fx,F^2x,...)$ embedds $(\X,F)$ into $(\X^\N,T)$, where $T$ is the shift map. In other words we have a $T$-invariant subset $Y\subset \X^\N$ where $(Y,T)$ is conjugate to $(\X,F)$.
The metric $d_T$ on $\X^\N$ restricted to $Y$ induces a metric on $\X$. We still denote this metric by $d_T$. More explicitely we have $d_T(x,y)=\sum_{k\ge 0} \frac{1}{2^k}d(T^kx,T^ky)$. Since $d(x,y)\le d_T(x,y)$ for all $x,y\in \X$ we conclude that 
\begin{align*} \overline{mdim}(\X,d,F)\le & \overline{mdim}(\X,d_T,F)=\overline{mdim}(Y,d_T,T)\\ \le & \overline{mdim}(\X^\N,d_T,T) =\overline{dim}_B(\X,d).
\end{align*}
In particular the dimension of $\X$ is always an upper bound for the mean dimension. 
\end{remark}
A map $F$ that realize the equality in the above inequality is said to have maximal metric mean dimension. We will construct an example of a continuous map on the interval with maximal metric mean dimension. From now on $d$ will always stand for the euclidean distance in $[0,1]$. We start with the following Lemma. 

\begin{lemma} \label{exam} Let $J_k=[a_{k-1},a_k]\subset [0,1]$ and $b_k=a_k-a_{k-1}$. Decompose $J_k$ into $2l_k$ equal intervals $\{J^s_k\}_{s=1}^{2l_k}$ and define $\e_k=\frac{b_k}{2l_k}$. Suppose that $f:[0,1]\to [0,1]$ is a continuous map such that $J_k\subset f(J^s_k)$. Then $$S([0,1],d,\e_k)\ge \log(l_k).$$
\end{lemma}
\begin{proof} Define the cylinders $$A_{i_0,...,i_{m-1}}=\{x\in J_k: f^{h}x\in J^{i_h}_s, \quad \textrm{for every } h\in\{0,...,m-1\}\}.$$
We will only consider those sets with all $i_h$'s odd. If $x$ and $y$ are in distinct cylinders, then $d_n(x,y)\ge e_k$. Moreover by the assumption $J_k\subset f(J^s_k)$ we know each cylinder is nonempty.  This implies that the maximal number of $(m,\e_k)$-separated points in $J_k$ is at least $(l_k)^m$ for each $m\ge 1$. In particular $$S([0,1],d,\e_k)\ge S(J_k,d,\e_k)\ge \log l_k.$$

\end{proof}

\begin{proposition} \label{existence} There exists $f:[0,1]\to [0,1]$ continuous such that $$\overline{mdim}([0,1],d,f)=1.$$
\end{proposition}
\begin{proof} 
Let  $b_k=C/k^2$ such that $\sum_{k\ge1} b_k=1$ and $l_k=k^k$. Define $a_k=\sum_{1\le s\le k} b_s$, $a_0=0$ and $J_k=[a_{k-1},a_k]$. It is easy to construct a continuous function $f$ such that the requirements in Lemma \ref{exam} are satisfied.   We can conclude that $$mdim([0,1],d,f)=\limsup_{\e\to 0} \frac{S([0,1],d,\e)}{|\log\e|}\ge \limsup_{k\to \infty} \frac{S([0,1],d,\e_k)}{|\log\e_k|}\ge  \limsup_{k\to \infty} \frac{\log l_k}{|\log\e_k|},$$
where $\e_k=\frac{b_k}{2l_k}=\frac{C}{2k^{k+2}}$. Then 
$$\overline{mdim}([0,1],d,f)\ge  \limsup_{k\to \infty} \frac{\log l_k}{|\log\e_k|} = \limsup_{k\to \infty} \frac{\log k^k}{\log (2k^{k+2}/C)}=1. $$
The opposite inequality holds for any continuous map by Remark \ref{remmm} since $$dim_B([0,1],d)=1.$$
\end{proof}

A similar construction allows us to prove that if $f(x)=x$ has infinitely many solutions, then $f$ can be approximated in the $C^0$ topology by continuous functions with metric mean dimension equal to one (slightly perturb the function $f$ nearby the fixed points by a map as in Lemma \ref{exam} using $l_k$ sufficiently large in comparison with $b_k$). This argument allows us to prove that the space $$A=\{F\in C([0,1]): \overline{mdim}([0,1],d,F)=1\},$$
where $C([0,1])$ is the space of continuous self maps on the interval $[0,1]$ with the uniform topology, is dense in $C([0,1])$. To see this start with $G\in C([0,1])$. By the mean value theorem  there exists $x_G\in [0,1]$ such that $G(x_G)=x_G$. Nearby $x_G$ perturb $G$ to a function $G'$ with infinitely many fixed points. Then mimic the construction done in Proposition \ref{existence} as mentioned above. We summarize this in the following result.
\begin{proposition} Let $$A=\{F\in C([0,1]): \overline{mdim}([0,1],d,F)=1\}.$$
Then $A$ is a dense subset of $C([0,1])$. 
\end{proposition}

A similar procedure can be done in a compact manifold $\X$ of dimension bigger or equal than two with a metric $d$ induced from a Riemannian metric $g$.  We will prove that any homeomorphism $F$ on $\X$ can be approximated by a homeomorphism with maximal metric mean dimension under the assumption that $F$ has a fixed point. For simplicity in the notation we will explain the procedure for a homeomorphism in a two manifold although the same argument holds in any dimension. Let $p$ be a fixed point of $F$. We make an arbitrary small perturbation $F_1$ of $F$ in a neighborhood of $p$ to create countably many isolated fixed points. We will perturb $F_1$ in a sufficiently small neighborhood of each fixed point (all neighborhood will be disjoint). As done in \cite{yan} we can perturb $F_1$ in a box of length $b_k$ (for arbitrarily small $b_k$) around the fixed point $p_k$ such that the perturbation contains a $2l_k$-horseshoe. Since our box is sufficiently small we can compare the metric $d$ with the euclidean metric up to multiplication by a uniform positive constant (here we are strongly using the compactness of $\X$). We will explain the local picture required in the perturbation. We start with $\R^2$ endowed with the euclidean metric $d^0$ and a square $R$ of lenght $b_k$. We define a map $G_k$ as shown in Figure 1. Let $\{U_i\}_{i=1}^{2l_k}$ be vertical strips of horizontal lenght $b_k/(2l_k)$ and the regions $\{A_i\}$ contain in $G_k(R)$ as in Figure 1.  Let $B_{ij}=G_k^{-1}(A_i\cap U_j)$ for $(i,j)$ a pair of odd numbers in $[1,2l_k]$. 

\begin{center}
  \begin{tikzpicture}
    \draw (0,0) -- (5,0) -- (5,5) -- (0,5) -- (0,0);
    \draw[semithick] (-1,4.7) -- (-1,4.4);
    \draw[black] plot coordinates {(-1,4.4) (6,4.4) (6,4.1) (-1.5,4.1) (-1.5,3.2) (6,3.2) (6, 2.9)};
    \draw[black] plot coordinates {(-1,4.7) (6.5,4.7) (6.5,3.8) (-1,3.8) (-1,3.5) (6.5,3.5) (6.5,2.9)};
    \node [below] at (2.3,-0.58) {$\textmd{Figure 1\label{fig1}: Horseshoe}$};
    \draw [dashed, ultra thick] (2.4,2.8) -- (3.4,2);
       
   \draw[black] plot coordinates {(-1.5,1.9) (-1.5,1.3) (6,1.3) (6, 1) (-1, 1) };
   \draw[black] plot coordinates {(-1,1.9) (-1,1.6) (6.5,1.6) (6.5,0.7) (-1, 0.7)};
   \draw[semithick] (-1, 1)  -- (-1, 0.7);

   \filldraw[draw=black,fill=gray!20]
    plot [domain=0:5] (\x,1)
    -- plot [smooth,domain=5:0] (\x,0.7)
    -- cycle;
   
   \filldraw[draw=black,fill=gray!20]
    plot [domain=0:5] (\x,1.6)
    -- plot [smooth,domain=5:0] (\x,1.3)
    -- cycle;
   \filldraw[draw=black,fill=gray!20]
    plot [domain=0:5] (\x,4.7)
    -- plot [smooth,domain=5:0] (\x,4.4)
    -- cycle;   
   \filldraw[draw=black,fill=gray!20]
    plot [domain=0:5] (\x,4.1)
    -- plot [smooth,domain=5:0] (\x,3.8)
    -- cycle;
   \filldraw[draw=black,fill=gray!20]
    plot [domain=0:5] (\x,3.2)
    -- plot [smooth,domain=5:0] (\x,3.5)
    -- cycle;         

    \draw[semithick] (0.5,0) -- (0.5,5);
    \draw[semithick] (1,0) -- (1,5);
    \draw[semithick] (1.5,0) -- (1.5,5);
    \draw[semithick] (2,0) -- (2,5);
    \draw[semithick] (0,0) -- (0,5);

    \draw[|-,semithick] (0,-0.2) -- (1.6,-0.2);
    \draw[-|,semithick] (3.6,-0.2) -- (5,-0.2);
    \draw [dashed, ultra thick] (1.7,-0.2) -- (3.5,-0.2);
    
    \draw[semithick] (0,-0.1) -- (0,-0.3);
    \draw[semithick] (0.5,-0.1) -- (0.5,-0.3);
    \draw[semithick] (1,-0.1) -- (1,-0.3);
    \draw[semithick] (1.5,-0.1) -- (1.5,-0.3);

    \node [below] at (0.25,-0.2) {$U_1$};
    \node [below] at (0.75,-0.2) {$U_2$};
    \node [below] at (1.25,-0.2) {$U_3$};

    \node [above] at (2.7,4.3) {$A_1$};
    \node [above] at (2.7,3.7) {$A_2$};
    \node [above] at (2.7,3.1) {$A_3$};
    \node [above] at (2.7,1.15) {$A_{2l_k-1}$};
    \node [above] at (2.7,0.55) {$A_{2l_k}$};

    \end{tikzpicture}
\end{center}    

Define the cylinders $$A_{(i_0j_0),...,(i_{m-1}j_{m-1})}=\{x\in R: G_k^s x\in B_{i_sj_s} \quad \textrm{for every } s\in\{1,...,m-1\}\}.$$ It follows from the construction that each cylinder is not empty. Moreover if $x$ and $y$ are in different cylinders then $d^0_m(x,y)>\e_k$, where $\e_k=b_k/(4l_k)$. We can conclude that $S(R,d^0, \e_k)\ge \log l_k^2$. We can perturb $F_1$ nearby the points in $\{p_k\}_{}$ by a map $F_2$ that nearby $p_k$ looks like $G_k$. The same reasoning as in Lemma \ref{exam} gives us that $$\overline{mdim}(\X,d,F_2)\ge \limsup_{k\to \infty} \dfrac{\log l_k^2}{\log (4l_k/b_k)}=\limsup_{k\to \infty}\dfrac{2}{1+\frac{\log 4}{\log l_k}+\frac{\log b_k^{-1}}{\log l_k}}.$$
Observe we can assume $b_k^{-1}\ll \log _k$. Therefore we found a map $F_2$ arbitrarily close to $F$ such that $\overline{mdim}(\X,d,F_2)=2$. For a general homeomorphism we can still approximate by homeomorphisms with positive upper metric mean dimension. To see this take $p\in \X$ such that $F^k(p)$ is sufficiently close to $p$, then perturb to create a fixed point using the standard closing lemma for homeomorphism. After this approximation is done we proceed as above or as done in \cite{yan}.  The same argument can be applied to prove that any homeomorphism $F$ with a $k$-periodic point can be approximated by homeomorphisms with upper metric mean dimension at least $\frac{1}{k}\dim(\X)$.  We denote by $H(\X)$ to the space of homeomorphisms of $\X$. We remark that under the assumptions above we know $\dim_B(\X,d)=dim(\X)$. We summarize this discussion in the following result. 

\begin{proposition} Let $(\X,g)$ be a closed Riemannian manifold with induced metric $d$. The set
$$A=\{F\in H(\X): \overline{mdim}(\X,d,F)=\dim(\X)\},$$
is dense in the subspace of homeomorphisms with a fixed point. The set
$$B=\{F\in H(\X): \overline{mdim}(\X,d,F)>0\},$$
is dense in $H(\X)$. 
\end{proposition}

\section{Final Remarks}\label{sec6}

For simplicity in previous sections we worked with the upper metric dimension. It also follows from Lemma \ref{1111} that if $\overline{mdim}(\X,d,T)<\infty$, then 
$$\underline{mdim}(\X,d,T)=\lim_{\delta\to 0}\liminf_{\e\to 0} \dfrac{\sup_\mu h_\mu(\e,T,\delta)}{|\log \e|}.$$ 
A  more symmetric formulation of the variational principle would be 
$$\overline{mdim}(\X,d,T)=\limsup_{\e\to 0} \dfrac{\sup_\delta \sup_\mu h_\mu(\e,T,\delta)}{|\log \e|},$$ 
and
$$\underline{mdim}(\X,d,T)=\liminf_{\e\to 0} \dfrac{\sup_\delta \sup_\mu h_\mu(\e,T,\delta)}{|\log \e|}.$$ 
This also follows directly from Lemma \ref{1111}. Analog statements are true for the average dynamical metric and the corresponding average metric mean dimension. 

The computation done in Theorem \ref{calculo} also shows that the metric mean dimension  depends strongly on the metric. If we have two compatible metrics $d_1$ and $d_2$ on a topological space $Y$ with different box dimension, then the metric mean dimension of the shift map $T$ on $Y^\Z$ will differ when considering the metric $(d_1)_T$ and $(d_2)_T$. 

It is an interesting question to describe hypothesis on the dynamics that ensure the existence of a measure of `maximal metric mean dimension'.  In our context this would mean that $$\overline{mdim}(X,d,T)=\lim_{\delta\to 0} \limsup_{\e\to 0} \dfrac{h_\mu(\e,T,\delta)}{|\log\e|}.$$
In the context of \cite{LT}, this would mean 
$$\overline{mdim}(X,d,T)=\limsup_{\e\to 0} \dfrac{R_\mu(\e)}{|\log\e|}.$$
As our bounds show, there is not major difference between this two perspectives (at least under Condition 1.2 we could link this two sides by using $\widetilde{h}_\mu(\e,T,\delta)$). The authors believe that if $\X$ is a manifold and $T$ preserves a OU measure  $\mu$ , i.e. $\mu$ has not atoms and it is positive on every open set, then $\mu$ would have this property (at least when choosing $T$ generic in the space of homeomorphisms that preserve $\mu$). To motivate the importance of this situation see \cite{ou}.

We define the \emph{upper metric mean dimension of a measure} $\mu$ as 
$$\overline{mdim}_\mu(\X,d,T)=\lim_{\delta\to 0}\limsup_{\e\to 0}\frac{h_\mu(\e,T,\delta)}{|\log\e|}.$$
As before, we could also use the analog formulation in the context of \cite{LT}. We remark that this quantity depends on the metric, it is not a purely measure theoretic invariant. As remarked in the introduction of \cite{LT}, it is not possible to define a meaningful invariant that only depends on measure theoretic information. Another natural question would be if the variational principle holds for this definition, in other words, if it is true that $\overline{mdim}(\X,d,T)=\sup_\mu \overline{mdim}_\mu(\X,d,T)$.

%\section{Metric mean dimension of a measure}\label{sec5}

%\begin{definition} We define the \emph{mean dimension of the ergodic measure} $\mu$ as %$$\overline{mdim}_\mu(\X,d,T)=\lim_{\delta\to 0}\limsup_{\e\to 0} \dfrac{h_\mu(\epsilon,T,\delta)}{|\log\e|}.$$
%\end{definition}
%If follows directly from the definition that if $h_\mu(T)<\infty$, then $\overline{mdim}_\mu(\X,d,T)=0$, in other words this definition is only meaningful for ergodic measures with infinite entropy.  Observe that $\overline{mdim}_\mu(\X,d,T)$ depends on $d$, same as in LY expression. 

%\begin{corollary} Under the hypothesis of Theorem \ref{teo1} we have the formula
%$$\overline{mdim}_\mu(\X,d,T)=\lim_{\epsilon\to 0} \dfrac{R_\mu(\epsilon)}{|\log\epsilon|} $$
%\end{corollary}

%\begin{corollary}  $$\overline{mdim}_\mu(\X,d,T)\le \overline{mdim}(\X,T)$$
%\end{corollary}
%But variational principle can not work in general as commented in [LY].

\end{document}